\documentclass[11pt]{article}
\usepackage[utf8]{inputenc}
\usepackage{amsmath,amssymb,amsthm}
\usepackage[margin=2.6cm]{geometry}
\usepackage{enumitem}
\usepackage{booktabs}
\usepackage{array}
\usepackage{tikz}
\usetikzlibrary{arrows.meta,calc}

\theoremstyle{plain}
\newtheorem{theorem}{Theorem}[section]
\newtheorem{proposition}[theorem]{Proposition}
\newtheorem{lemma}[theorem]{Lemma}
\newtheorem{corollary}[theorem]{Corollary}
\theoremstyle{definition}
\newtheorem{definition}[theorem]{Definition}
\newtheorem{countermodel}[theorem]{Countermodel}
\newtheorem{example}[theorem]{Example}
\theoremstyle{remark}
\newtheorem{remark}[theorem]{Remark}

\newcommand{\SDS}{\mathsf{SDS}}
\newcommand{\Aut}{\operatorname{Aut}}
\newcommand{\rank}{\operatorname{rank}}
\newcommand{\supp}{\operatorname{supp}}
\newcommand{\cyc}{c}
\newcommand{\Sbar}{\overline{\Sigma}}
\newcommand{\Dbar}{\overline{\Delta}}
\newcommand{\rresp}{r_{\mathrm{resp}}}

\title{Axioms and internal rigidity of spectral duality structures}
\author{J.~A.~Vega Coso \\
        \small IUFFyM, Universidad de Salamanca \\
        \small Salamanca, Spain}
\date{}

\begin{document}
\maketitle

\begin{abstract}
Spectral duality structures (SDS) were introduced to explain why a
distinguished family of initial distributions makes the absorption
probabilities of resetting processes insensitive to the resetting mechanism. Their algebraic theory was developed in a previous paper.
Here we study the axioms themselves. We prove that the six defining
conditions are mutually independent. For the first one, that the set
of sites be finite and non-empty, independence rests on
non-emptiness: finiteness follows from the full-rank condition for
complete structures. We then determine what each condition
contributes. The determination rests on a structure theorem for
insensitive distributions that holds for an arbitrary permutation and
an arbitrary positive weight: every insensitive distribution is
supported on a union of cycles balanced at a common level, and the
level of a cycle is the geometric mean of the weight on it. From this
theorem the role of each axiom can be read off directly. Positivity
selects the admissible root. The constant orbit product is equivalent
to the existence of an insensitive distribution of full support. The
involution is equivalent to the equality between the response rank
and the number of free weight parameters. Without the duality
relation, the insensitivity of a single distribution is cheap and,
when the number of modes exceeds the number of sites, generically
isolated; together with the other axioms, duality is what
turns it into the structural neutrality of the separatrix. The full-rank
condition splits into two layers: a weaker condition on an explicit
matrix $M$ is equivalent both to the response rank being maximal and
to the separatrix being the whole $C^\ast$-level set, while full rank,
under the remaining axioms, additionally rules out insensitivity at
any other level. In the last
part we show that a single SDS is internally rigid. A compatible
semigroup structure forces the weight to be constant, abelian group
structures compatible with the involution obey a precise arithmetic
condition, quotients are governed by a simple lattice of congruences,
and the automorphism group is computed explicitly. We also show that
the axioms do not determine the orientation of the response: in the
two-site case we give the exact condition under which the geometric
and analytic sign criteria agree.
\end{abstract}


\section{Introduction}
\label{sec:intro}

Stochastic resetting, in which a process is interrupted and restarted
from a prescribed distribution, has become a standard device for
modelling search and first-passage problems \cite{EMS}. In a series
of papers on random walks with geometric resetting
\cite{PaperI,PaperII,PaperIII,PaperIV}, the author found that for a
distinguished family of initial distributions the probability of
leaving an interval through a given end does not depend on the
resetting parameter --- and, as Paper~I shows, not even on the
resetting mechanism --- and that its common value is
$C^\ast=1/(1+\sqrt K)$, for a constant $K$ determined by the walk. The
explanation turned out to be algebraic rather than probabilistic. The
two exit channels admit spectral representations whose amplitudes are
related by a reflection of the sites and by a positive weight that does
not depend on the spectral mode. Combined with the corresponding
balance condition on the initial distribution, this mode-independent
duality forces the insensitivity.

Paper~V \cite{PaperV} distilled the mechanism into an abstract notion.
A \emph{spectral duality structure} (SDS) consists of a finite set of
sites $Z$, an involution $\sigma$ of $Z$, a positive weight $\kappa$
whose product over each orbit is a constant $K$, and a family of
spectral amplitudes $A_\nu$ related to a second family $B_\nu$ by the
duality $B_\nu(z)=\kappa(z)A_\nu(\sigma z)$. Paper~VI \cite{PaperVI}
developed the algebra of the \emph{class} of SDS: a parametrisation of
the weights by $(K,r)$, Cartesian product and disjoint union, a
category with a graded symmetric monoidal structure, a classification
up to isomorphism, and the response rank theorem.

The present paper asks a different question. It is not concerned with
what can be built from SDS, but with what is \emph{inside} one: why
the axioms have the form they have, whether any of them is redundant,
what each of them buys, and how much internal structure a single
object can carry. This is the kind of foundational analysis that
should accompany any new axiomatic notion, and it has four parts. The
first is \emph{independence}: whether any axiom follows from the
others. The second is the \emph{role} of each axiom, which is a
different question, since an axiom can be independent and still idle,
or independent and stronger than what the theorems use. The third is
\emph{characterisation}: what the axioms determine jointly, and what
they leave undetermined. The fourth is the \emph{internal structure} of
a single object: which operations, quotients and symmetries it admits.
The analysis turns out to be more informative than one might expect.

\paragraph{Method.} The key step is a reduction lemma, which turns
the insensitivity of a distribution, a functional identity in the
spectral parameter, into a finite-dimensional linear condition. We
solve that condition not within the axioms but for an arbitrary
permutation of the sites and an arbitrary positive weight. This makes
it possible to switch the axioms on and off one at a time and to see
exactly what each of them changes. Figure~\ref{fig:axioms} summarises
the resulting picture.

\paragraph{Thesis.} The axioms are independent, and each one does an
identifiable piece of work. That work, however, is not always where
the original derivation placed it. The mechanism of insensitivity uses
neither the involution nor the pointwise constancy of the orbit
product: for any permutation and any positive weight, a balanced cycle
is insensitive at the level $1/(1+g_O)$ fixed by the geometric mean
$g_O$ of the weight on it. The involution and the constant product
then do two separate things. Together they make all these levels
coincide, with common value $C^\ast$; and the involution, given the
pointwise product condition, is what makes the \emph{response} of the
structure match its \emph{moduli}. A single SDS, finally, is rigid:
it admits essentially no compatible internal operations, its
quotients and symmetries are completely described by its orbit data,
and the interesting structure lies in the class, which is the object
of \cite{PaperVI} and of subsequent work.

\paragraph{Main results.}
\begin{enumerate}[leftmargin=*]
\item \emph{Independence} (Theorem~\ref{thm:independence}): none of
  the conditions (A1)--(A4), (B1), (B2) follows from the others.
\item \emph{Structure of insensitive distributions}
  (Theorem~\ref{thm:master}): for any permutation $\sigma$ and any
  positive weight, under full rank and duality, a distribution is
  insensitive to the spectral parameter at level $c$ if and only if
  it is supported on a union of cycles of geometric mean
  $(1-c)/c$, on each of which it satisfies a balance relation.
\item \emph{Role of each axiom} (Section~\ref{sec:roles}):
  \begin{itemize}[nosep]
  \item (A3) selects the positive root and places all levels
    in $(0,1)$.
  \item (A4) is equivalent to the existence of an insensitive
    distribution of full support (Corollary~\ref{cor:A4}).
  \item (A2) is equivalent to the equality between the response rank
    and the number of weight moduli (Corollary~\ref{cor:A2}).
  \item without (B2), the insensitivity of a single distribution is
    cheap and, for $N>m$, generically isolated (Propositions~\ref{prop:cheap}
    and~\ref{prop:isolated}); structural neutrality needs (B2)
    together with the other axioms
    (Theorem~\ref{thm:characterisation}).
  \item (B1) splits into two layers (Theorems~\ref{thm:rankM}
    and~\ref{thm:levelset}).
  \end{itemize}
\item \emph{Characterisation} (Theorem~\ref{thm:characterisation}):
  under all the axioms, the insensitive distributions are exactly the
  separatrix $\Sbar$, of dimension $f+k-1$, and the only insensitive
  value is $C^\ast$.
\item \emph{Orientation} (Section~\ref{sec:orientation}): the
  axioms do not fix the sign of the response; in the two-site case the
  geometric and analytic criteria agree exactly where the structural
  ratio $R$ moves away from $1$, and disagree where it moves towards
  it (Theorem~\ref{thm:orient}).
\item \emph{Internal rigidity} (Section~\ref{sec:rigidity}): weight
  compatible semigroup structures force $\kappa$ to be constant; the
  arithmetic of compatible abelian group structures is determined
  exactly; quotients are classified by congruences; and the
  automorphism group is computed.
\end{enumerate}

\paragraph{Organisation.} Section~\ref{sec:setting} fixes the axioms,
the analytic setting and a reduction lemma.
Section~\ref{sec:independence} proves independence.
Section~\ref{sec:roles} determines the role of each axiom and contains
the structure theorem. Section~\ref{sec:characterisation} states what
the axioms characterise and what they do not.
Section~\ref{sec:orientation} shows that the orientation of the
response is not determined by the axioms.
Section~\ref{sec:rigidity} studies the internal structure of a single
skeleton. Section~\ref{sec:conclusions} concludes, and
Section~\ref{sec:further} lists further directions.

\section{Setting}
\label{sec:setting}

\subsection{The axioms}

We follow the definitions of \cite[Definitions~2.1 and~2.3]{PaperVI},
separated into elementary conditions. It is convenient to regard both
families of amplitudes as data, and the duality as a relation between
them rather than as a definition of the second family. This is how they
arise in the realisations, where $A_\nu$ and $B_\nu$ are the spectral
amplitudes of two distinct absorption channels \cite{PaperIII}. It is
also the only formulation in which the question of the independence of
the duality relation makes sense.

\begin{definition}[Axioms]
\label{def:axioms}
Consider data $(Z,\sigma,\kappa,K,\{A_\nu\}_{\nu=1}^{N},
\{B_\nu\}_{\nu=1}^{N})$, where $Z$ is a set, $\sigma:Z\to Z$ is a map,
$\kappa:Z\to\mathbb R$, $K\in\mathbb R$, and $A_\nu,B_\nu:Z\to\mathbb R$.
Consider the conditions:
\begin{description}[leftmargin=3.2em,style=nextline]
\item[(A1)] $Z$ is finite and non-empty.
\item[(A2)] $\sigma$ is an involution: $\sigma^2=\mathrm{id}_Z$.
\item[(A3)] $\kappa(z)>0$ for every $z\in Z$.
\item[(A4)] $\kappa(z)\,\kappa(\sigma z)=K$ for every $z\in Z$.
\item[(B1)] The vectors
  $\mathbf A(z)=(A_1(z),\dots,A_N(z))\in\mathbb R^N$, $z\in Z$, are
  linearly independent.
\item[(B2)] $B_\nu(z)=\kappa(z)\,A_\nu(\sigma z)$ for every $\nu$ and
  every $z$.
\end{description}
A \emph{weighted dual skeleton} is a triple $\mathcal E=(Z,\sigma,\kappa)$
satisfying (A1)--(A4). A \emph{complete SDS} is a skeleton together
with amplitudes satisfying (B1) and (B2).
\end{definition}

Under (A3), the constant $K$ in (A4) is automatically positive. We
write $m=|Z|=2k+f$, where $k$ is the number of non-trivial orbits of
$\sigma$ and $f$ the number of fixed points. We write
$r_z=(\kappa(\sigma z)/\kappa(z))^{1/2}$ for the orbit ratios, so
that, by \cite[Proposition~3.1]{PaperVI},
$\kappa(z)=\sqrt K/r_z$ and $\kappa(\sigma z)=\sqrt K\,r_z$, with
$r_z=1$ and $\kappa(z)=\sqrt K$ at fixed points. Since
$r_{\sigma z}=1/r_z$, for a non-trivial orbit $O$ we write $r_O$ for the
ratio at a chosen representative; the other choice gives $1/r_O$, so
the class $[r_O]=\{r_O,r_O^{-1}\}$ is intrinsic. The universal
invariant is $C^\ast=1/(1+\sqrt K)$. We write
$\Dbar=\{\pi\in[0,1]^Z:\sum_z\pi_z=1\}$ for the closed simplex of
distributions on $Z$, and $\Delta$ for its interior, the open simplex
of distributions with all $\pi_z>0$.

In Section~\ref{sec:roles} we shall drop some axioms. To keep the
statements meaningful in that case, we always assume that $\sigma$ is
at least a \emph{bijection}, and we use the following notation. For a
bijection $\sigma$ of a finite set, $\cyc(\sigma)$ denotes the number
of its cycles (fixed points included) and $e(\sigma)$ the number of
its cycles of even length. For a cycle $O$ and a positive weight
$\kappa$ we put
\[
g_O=\Bigl(\prod_{z\in O}\kappa(z)\Bigr)^{1/|O|},
\]
the geometric mean of $\kappa$ on $O$. If $\sigma$ is an involution,
then $g_O=\sqrt{\kappa(z)\kappa(\sigma z)}$ on a non-trivial orbit and
$g_O=\kappa(z)$ at a fixed point. Under (A2)--(A4), therefore,
$g_O=\sqrt K$ for every orbit.

\subsection{The analytic setting}

Let $f_1,\dots,f_N:(0,1)\to\mathbb R$ be the spectral functions. In
the realisations, up to a common nonvanishing factor,
$f_\nu(\gamma)=(1-\gamma)/[1-\lambda_\nu(1-\gamma)]$, with eigenvalues
$\lambda_\nu$. Nothing below uses the probabilistic meaning of
$\gamma$: it is an abstract spectral parameter, and the results hold
for any family $f_\nu$ satisfying (H1). Geometric resetting is one
realisation among others. Following \cite[Section~8]{PaperVI}, put
\begin{equation}
u_z(\gamma)=\sum_{\nu}f_\nu(\gamma)A_\nu(z),
\qquad
s_z(\gamma)=\sum_{\nu}f_\nu(\gamma)\bigl[A_\nu(z)+B_\nu(z)\bigr],
\label{eq:us}
\end{equation}
so that under (B2), $s_z=u_z+\kappa(z)u_{\sigma z}$. For a probability
vector $\pi$ on $Z$, the \emph{coupling functional} is
\begin{equation}
C(\pi,\gamma)=\frac{\langle\pi,u(\gamma)\rangle}{\langle\pi,s(\gamma)\rangle},
\qquad \langle\pi,x\rangle=\sum_z\pi_zx_z .
\label{eq:coupling}
\end{equation}
The analytic results below use two standing hypotheses, which are
hypotheses of theorems and not axioms of the structure:
\begin{description}[leftmargin=3.2em,style=nextline]
\item[(H1)] $f_1,\dots,f_N$ are linearly independent as functions on
  $(0,1)$.
\item[(H2)] $s_z(\gamma)>0$ for every $z$ and every $\gamma$.
\end{description}
Hypothesis (H1) is automatic in the realisations when the eigenvalues
are distinct \cite[Remark~8.13]{PaperVI}. Hypothesis (H2) makes
\eqref{eq:coupling} well defined on the closed simplex $\Dbar$, and
holds in the realisations, where $s_z(\gamma)$ is the probability of
absorption before the next reset, starting from $z$.

\begin{definition}[Insensitive and neutral distributions]
\label{def:insensitive}
A distribution $\pi\in\Dbar$ is \emph{insensitive at level $c$} if
$C(\pi,\gamma)=c$ for every $\gamma\in(0,1)$. For a skeleton, the
\emph{(closed) separatrix} is
\[
\Sbar=\bigl\{\pi\in\Dbar:\ \sqrt K\,\pi_{\sigma z}=\kappa(z)\,\pi_z
  \ \text{ for every } z\in Z\bigr\}.
\]
This is the linear form of the neutrality condition
$\pi_z/\pi_{\sigma z}=r_z$ of \cite[Definition~8.2]{PaperVI}. Its
elements are called \emph{neutral}.
\end{definition}

By \cite[Lemma~8.3]{PaperVI}, every neutral distribution is insensitive
at level $C^\ast$. The converse is the content of
Section~\ref{sec:characterisation}. We recall the definition of the
response rank. Fix $\pi^\ast$ neutral in the open simplex and put
$\psi(\gamma)=(u(\gamma)-C^\ast s(\gamma))/\langle\pi^\ast,s(\gamma)\rangle$,
the gradient of $C(\cdot,\gamma)$ at $\pi^\ast$. Then
$\rresp=\dim\operatorname{span}\{\psi(\gamma):\gamma\in(0,1)\}$, the
span being taken in $\mathbb R^Z$ as in
\cite[Remark~8.5]{PaperVI}. The same definition makes sense whenever
an interior distribution is insensitive at level $C^\ast$, which is
how we use it for bijections in Section~\ref{subsec:A2}.

\subsection{A reduction lemma}

Write $T:\mathbb R^Z\to\mathbb R^N$, $Tx=\sum_zx_z\mathbf A(z)$, for
the amplitude map. Condition (B1) says exactly that $T$ is injective.
For $c\in\mathbb R$, let $L_c:\mathbb R^Z\to\mathbb R^Z$ be the linear
map
\begin{equation}
(L_cx)_{\sigma z}=(1-c)\,x_{\sigma z}-c\,\kappa(z)\,x_z,
\qquad z\in Z,
\label{eq:Lc}
\end{equation}
which is well defined when $\sigma$ is a bijection. Equivalently,
$(L_cx)_z=(1-c)x_z-c\,\kappa(\sigma^{-1}z)\,x_{\sigma^{-1}z}$; for an
involution this reads $(L_cx)_z=(1-c)x_z-c\,\kappa(\sigma z)x_{\sigma
z}$. We use the first form because it makes the balance relations
below read directly along the cycle.

\begin{lemma}[Reduction]
\label{lem:reduction}
Let $Z$ be finite, $\sigma$ a bijection, and assume \textup{(B2)},
\textup{(H1)} and \textup{(H2)}. Then $\pi\in\Dbar$ is insensitive at
level $c$ if and only if $T L_c\pi=0$.
\end{lemma}

\begin{proof}
$C(\pi,\gamma)=c$ for every $\gamma$ is equivalent to
$\langle\pi,u(\gamma)-c\,s(\gamma)\rangle=0$ for every $\gamma$. By
\eqref{eq:us} and (B2), this is
\[
\sum_\nu f_\nu(\gamma)\sum_z\pi_z\Bigl[(1-c)A_\nu(z)-c\,\kappa(z)
A_\nu(\sigma z)\Bigr]=0 .
\]
By (H1) the coefficient of each $f_\nu$ vanishes, so the condition is
equivalent to the vector identity
$\sum_z\pi_z[(1-c)\mathbf A(z)-c\,\kappa(z)\mathbf A(\sigma z)]=0$.
Since $\sigma$ is a bijection, we may collect the coefficient of
$\mathbf A(w)$ for $w=\sigma z$, and the left-hand side becomes
$\sum_{z}\bigl[(1-c)\pi_{\sigma z}-c\,\kappa(z)\pi_z\bigr]\mathbf
A(\sigma z)=T L_c\pi$.
\end{proof}

\begin{remark}
\label{rem:modewise}
The lemma is the abstract form of the mode-by-mode reduction. The
insensitivity problem, a functional identity in $\gamma$, becomes a
finite-dimensional linear problem. No probabilistic structure is
involved. What enters is a spectral representation \eqref{eq:us}, the
linear independence (H1), and a duality (B2) in which the weight
$\kappa$ does not depend on the mode $\nu$. The last point is
essential: it is what allows a single distribution to balance all
modes at once.
\end{remark}

\section{Independence of the axioms}
\label{sec:independence}

\begin{theorem}[Independence]
\label{thm:independence}
None of the conditions \textup{(A1)}, \textup{(A2)}, \textup{(A3)},
\textup{(A4)}, \textup{(B1)}, \textup{(B2)} is a consequence of the
others.
\end{theorem}

The proof consists of the following countermodels. Each one satisfies
all the conditions except the one indicated, which proves that this
condition cannot be derived from the others. In the countermodels for
(A1)--(A4) the amplitudes can be taken to be any family satisfying
(B1), with $B$ defined by (B2), so we only describe the skeleton.

\begin{countermodel}[(A1)]
\label{cm:A1empty}
$Z=\varnothing$. Conditions (A2)--(A4), (B1) and (B2) hold vacuously,
with any $K>0$. The empty structure would be the initial object
of each stratum $\SDS_K$ \cite[Remark~5.6]{PaperVI}; excluding it is a
convention, and it is the only axiom of which this can be said.
\end{countermodel}

\begin{remark}[Finiteness is a consequence of (B1)]
\label{rem:finiteness}
The countermodel for (A1) uses its non-emptiness half. The finiteness
half is \emph{not} independent: since $N$ is a positive integer,
(B1) requires $|Z|$ linearly independent vectors in $\mathbb R^N$,
hence $|Z|\le N$. Finiteness is independent only at the level of
skeletons, i.e.\ of (A2)--(A4): $Z=\mathbb Z$, $\sigma(z)=-z$,
$\kappa\equiv1$ satisfies them. For complete structures it could be
dropped from (A1) without loss; we keep it because the skeleton is
the object of the categorical theory \cite[Note~2.5]{PaperVI}, and
there it is a genuine axiom.
\end{remark}

\begin{countermodel}[(A2)]
\label{cm:A2}
$Z=\{0,1,2\}$, $\sigma=(0\,1\,2)$, $\kappa\equiv1$, $K=1$. Then
$\kappa>0$ and $\kappa(z)\kappa(\sigma z)=1$ for all $z$, but
$\sigma^2(0)=2\neq0$.
\end{countermodel}

\begin{countermodel}[(A3)]
\label{cm:A3}
$Z=\{0,1\}$, $\sigma=(0\,1)$, $\kappa(0)=-2$, $\kappa(1)=-3$, $K=6$.
Then $\sigma^2=\mathrm{id}$ and $\kappa(0)\kappa(1)=6$, but $\kappa<0$.
\end{countermodel}

\begin{countermodel}[(A4)]
\label{cm:A4}
$Z=\{0,1,2,3\}$, $\sigma=(0\,1)(2\,3)$, $\kappa=(2,3,1,1)$. Then
$\sigma^2=\mathrm{id}$ and $\kappa>0$, but the orbit products are $6$
and $1$.
\end{countermodel}

\begin{countermodel}[(B1)]
\label{cm:B1}
Any skeleton with $m\ge2$, amplitudes with
$\mathbf A(z_1)=\mathbf A(z_2)$ for two distinct sites, and $B$ defined
by (B2).
\end{countermodel}

\begin{countermodel}[(B2)]
\label{cm:B2}
Any skeleton, amplitudes satisfying (B1) with $N\ge2$, and $B$ chosen
so that some vector $\mathbf B(z)$ is not proportional to any
$\mathbf A(w)$. Under (B2), every $\mathbf B(z)$ is proportional to
$\mathbf A(\sigma z)$.
\end{countermodel}

\begin{remark}[Independent is not the same as idle]
\label{rem:notidle}
The countermodels prove that no axiom can be \emph{derived} from the
others. They say nothing about what each axiom is \emph{used for}, and
the two questions have different answers. Countermodel~\ref{cm:A2} is
instructive. Its permutation is not an involution, yet with amplitudes
satisfying (B1) and $B$ defined by (B2), the uniform distribution is
insensitive at level $C^\ast=1/2$ (Theorem~\ref{thm:master} below). The
invariance mechanism survives the failure of (A2). Determining what
(A2) does is therefore a separate question, answered in
Section~\ref{subsec:A2}. Countermodel~\ref{cm:B2}, by contrast, does
lose the invariance, and in a strong sense. Take one orbit
$\{a,b\}$ and one fixed point $c$, with $\kappa=(1,4,2)$, so that
$K=4$ and $C^\ast=\tfrac13$; take $A=I_3$, spectral functions
$f_\nu(\gamma)=(1-\gamma)/[1-\lambda_\nu(1-\gamma)]$ with
$\lambda=(\tfrac12,-\tfrac14,\tfrac14)$, and the neutral distribution
$\pi^\ast=(\tfrac12,\tfrac14,\tfrac14)$. With $B$ given by (B2),
$C(\pi^\ast,\gamma)=\tfrac13$ exactly at $\gamma=\tfrac15,\tfrac12,\tfrac45$.
If $B$ is replaced by the all-ones matrix, which is positive and not of
the form (B2), the three values become $\tfrac{13}{49}$,
$\tfrac{37}{143}$ and $\tfrac{253}{999}$: they are not constant, and
none equals $C^\ast$.
\end{remark}

\section{What each axiom does}
\label{sec:roles}

\subsection{The structure theorem}

The following theorem is the tool for the whole section. It holds
without (A2) and without (A4).

\begin{theorem}[Structure of insensitive distributions]
\label{thm:master}
Let $Z$ be finite and non-empty, $\sigma$ a bijection of $Z$, and
$\kappa:Z\to(0,\infty)$ arbitrary. Assume \textup{(B1)},
\textup{(B2)}, \textup{(H1)} and \textup{(H2)}. No distribution is
insensitive at level $c=0$ or $c=1$. Let $\pi\in\Dbar$ and
$c\in\mathbb R\setminus\{0,1\}$. Then $\pi$ is insensitive at level $c$
if and only if the following two conditions hold:
\begin{enumerate}[label=\textup{(\roman*)}]
\item every cycle $O$ of $\sigma$ that meets $\supp\pi$ is contained in
  $\supp\pi$ and satisfies $g_O=(1-c)/c$;
\item on every such cycle, $g_O\,\pi_{\sigma z}=\kappa(z)\,\pi_z$ for
  all $z\in O$.
\end{enumerate}
In particular, every insensitive level is of the form
$c=1/(1+g_O)$ for some cycle $O$. The set of distributions
insensitive at level $c$ is either empty or a simplex of dimension
$n_c-1$, where $n_c$ is the number of cycles with
$g_O=(1-c)/c$.
\end{theorem}

\begin{proof}
By (B1), $T$ is injective, so by Lemma~\ref{lem:reduction}, $\pi$ is
insensitive at level $c$ if and only if $L_c\pi=0$, that is,
\begin{equation}
(1-c)\,\pi_{\sigma z}=c\,\kappa(z)\,\pi_z\qquad\text{for every }z\in Z.
\label{eq:balance}
\end{equation}

Suppose \eqref{eq:balance} holds, and let $\pi_z>0$ for some $z$ in a
cycle $O$ of length $\ell$. If $c=1$, then \eqref{eq:balance} at $z$
gives $\kappa(z)\pi_z=0$, which is impossible. If $c=0$, it gives
$\pi_{\sigma w}=0$ for every $w$, hence $\pi=0$, which is also
impossible. So $c\notin\{0,1\}$, and \eqref{eq:balance} reads
$\pi_{\sigma z}=\frac{c}{1-c}\,\kappa(z)\,\pi_z$. Since
$\pi_{\sigma z}\ge0$, $\kappa(z)>0$ and $\pi_z>0$, we get
$c/(1-c)\ge0$, and the case $c/(1-c)=0$ was excluded, so
$c/(1-c)>0$ and $\pi_{\sigma z}>0$. Iterating along the cycle shows
that $\pi$ is positive on all of $O$. Going once around the cycle
gives
\[
\pi_z=\Bigl(\frac{c}{1-c}\Bigr)^{\ell}\prod_{w\in O}\kappa(w)\;\pi_z ,
\qquad\text{hence}\qquad
\Bigl(\frac{1-c}{c}\Bigr)^{\ell}=g_O^{\ell}.
\]
Since both $(1-c)/c$ and $g_O$ are positive, $(1-c)/c=g_O$. This is
(i). Substituting into \eqref{eq:balance} gives (ii).

Conversely, assume (i) and (ii). If $z$ lies in a cycle contained in
$\supp\pi$, then $1-c=c\,g_O$, and \eqref{eq:balance} is (ii)
multiplied by $c$. If $z$ lies in a cycle disjoint from $\supp\pi$,
both sides of \eqref{eq:balance} vanish. Hence $L_c\pi=0$.

For the last assertion, note that on a cycle with $g_O=(1-c)/c$,
relation (ii) determines $\pi|_O$ up to a positive factor (the
consistency condition around the cycle is exactly
$\prod_O\kappa=g_O^{\ell}$). So the insensitive distributions at
level $c$ are the convex combinations of the $n_c$ normalised cycle
vectors.
\end{proof}

\begin{remark}[The level of a cycle]
\label{rem:level}
Theorem~\ref{thm:master} assigns to every cycle $O$ an intrinsic
number, its \emph{level} $c_O=1/(1+g_O)$. Insensitivity is a local
phenomenon: a distribution is insensitive exactly when all the cycles
it charges sit at the same level and are internally balanced. The
universal invariant $C^\ast=1/(1+\sqrt K)$ is the level of every orbit
of a skeleton, since $g_O=\sqrt K$ under (A2)--(A4).
\end{remark}

\subsection{The constant orbit product (A4)}
\label{subsec:A4}

\begin{corollary}[(A4) is equivalent to an insensitive distribution of full support]
\label{cor:A4}
Under the hypotheses of Theorem~\ref{thm:master}, there exists an
insensitive distribution in the open simplex if and only if all
cycles of $\sigma$ have the same geometric mean $g_O$. If moreover
$\sigma$ is an involution, this happens if and only if \textup{(A4)}
holds, and then the level is $C^\ast$.
\end{corollary}

\begin{proof}
An interior distribution charges every cycle, so by
Theorem~\ref{thm:master} all $g_O$ coincide. Conversely, if they all
equal $g$, the distribution that is balanced on each cycle with
arbitrary positive masses is insensitive at level $1/(1+g)$. If
$\sigma$ is an involution, $g_O^2=\kappa(z)\kappa(\sigma z)$ on a
non-trivial orbit and $g_O^2=\kappa(z)^2=\kappa(z)\kappa(\sigma z)$ at
a fixed point. So all $g_O$ are equal if and only if
$\kappa(z)\kappa(\sigma z)$ is constant, which is (A4) with
$K=g^2$.
\end{proof}

\begin{remark}[Every pair balances; the axiom lives between pairs]
\label{rem:pair}
Consider a single non-trivial orbit $\{z,z'\}$ with arbitrary weights
$\kappa(z),\kappa(z')>0$. By Theorem~\ref{thm:master} it is always
balanced: the distribution with $\pi_z/\pi_{z'}=
\sqrt{\kappa(z')/\kappa(z)}$ is insensitive at level
$1/(1+\sqrt{\kappa(z)\kappa(z')})$. No condition on the product is
involved. Condition (A4) is therefore not a condition on pairs. It is
the condition that all pairs, and all fixed points, \emph{share} a
level. This is what turns the separate balances of the orbits into a
single separatrix meeting the interior. Without (A4), the insensitive
distributions form simplices lying in pairwise disjoint proper faces
of the simplex, one for each value of the orbit product, and none of
them meets the interior. For Countermodel~\ref{cm:A4} each of them is
a single point: the balanced distribution on the first orbit, at level
$1/(1+\sqrt6)$, and the balanced distribution on the second, at level
$1/2$.
\end{remark}

\begin{remark}[(A4) for permutations]
For a bijection that is not an involution, ``all $g_O$ equal'' is
strictly weaker than the pointwise condition
$\kappa(z)\kappa(\sigma z)=K$. For instance, the $3$-cycle with
$\kappa=(1,2,\tfrac12)$ has $g_O=1$, so the balanced distribution
$\pi\propto(1,1,2)$ is insensitive at level $\tfrac12$, while the
products $\kappa(z)\kappa(\sigma z)$ are $2$, $1$ and $\tfrac12$. Thus
the pointwise form of (A4) is the right one precisely because $\sigma$
is an involution. Section~\ref{subsec:A2} shows that, for
non-involutive permutations, the pointwise form has further
consequences that the mechanism does not need.
\end{remark}

\subsection{Positivity (A3)}
\label{subsec:A3}

The role of (A3) appears twice. The first is the one recorded in
\cite[Remark~2.2]{PaperVI}: at a fixed point (A4)
reduces to $\kappa(z)^2=K$, and (A3) selects $\kappa(z)=+\sqrt K$. On
this rest the parametrisation $\kappa=(\sqrt K/r,\sqrt K\,r)$
\cite[Proposition~3.1]{PaperVI}, the classification
\cite[Theorem~7.2]{PaperVI} and the definition of $C^\ast$ itself.
Countermodel~\ref{cm:A3} shows that, without (A3), the same equation
admits the negative root.

The second is visible in the proof of Theorem~\ref{thm:master}. The
sign condition $c/(1-c)>0$, which forces every insensitive level into
$(0,1)$ and selects the \emph{positive} geometric mean, uses
$\kappa>0$. If $\kappa$ is negative on a cycle, the same computation
gives $(1-c)/c<0$ on it. Such a cycle can only be balanced at a level
outside $(0,1)$, never at the level of the positive cycles. In
Countermodel~\ref{cm:A3}, for example, the only candidate level for
the orbit is $c=1/(1-\sqrt6)<0$. Positivity is thus what makes ``the
level of a cycle'' a well-defined number in $(0,1)$, the same for
fixed points and for pairs.

\begin{remark}[Signed levels]
\label{rem:signed}
If signed vectors are allowed in place of distributions, the
computation in the proof of Theorem~\ref{thm:master} also admits the
root $(1-c)/c=-g_O$ on cycles of even length. If $g_O\neq1$ this is
the level $1/(1-g_O)$, with $\pi$ alternating in sign along the cycle;
if $g_O=1$ the equation $1-c=-c$ has no solution and this branch gives
no level.
Positivity of $\pi$ is what excludes it, just as positivity of
$\kappa$ excludes negative levels.
\end{remark}

\subsection{The involution (A2)}
\label{subsec:A2}

Theorem~\ref{thm:master} and Remark~\ref{rem:notidle} show that the
invariance mechanism does not need (A2). To find what (A2) does, we
keep the pointwise product condition and let $\sigma$ be an arbitrary
bijection.

\begin{proposition}[Weights on a permutation]
\label{prop:permweights}
Let $\sigma$ be a bijection of a finite set $Z$, and let $\kappa>0$
satisfy $\kappa(z)\kappa(\sigma z)=K$ for every $z$. Then
$\kappa\circ\sigma^2=\kappa$. On every cycle of odd length,
$\kappa\equiv\sqrt K$. On every cycle of even length, $\kappa$
alternates between two values $a$ and $K/a$ with $a>0$ arbitrary.
Consequently, for fixed $(Z,\sigma,K)$ the set of admissible weights is
in bijection with $(\mathbb R_{>0})^{e(\sigma)}$, and every cycle has
$g_O=\sqrt K$.
\end{proposition}

\begin{proof}
We have $\kappa(\sigma^2z)=K/\kappa(\sigma z)=\kappa(z)$. On a
cycle $z_0$, $z_1=\sigma z_0$, \dots, the values are therefore
$a$, $K/a$, $a$, $K/a$, \dots, with $a=\kappa(z_0)$. If the length is odd, closing
the cycle forces $a=K/a$, so $a=\sqrt K$ by positivity. If it is even,
the cycle closes for every $a>0$. The geometric mean is $\sqrt K$ in
both cases.
\end{proof}

For an involution, $e(\sigma)=k$ and the proposition reduces to
\cite[Proposition~3.1]{PaperVI}: one free parameter $r_O$ for each
non-trivial orbit. For a general permutation, the free parameters are
the even cycles, while odd cycles of length at least $3$ are rigid,
like fixed points.

\begin{theorem}[Separatrix and response rank for a permutation]
\label{thm:perm}
Let $\sigma$ be a bijection of a finite set $Z$ with $|Z|=m$, let
$\kappa>0$ satisfy $\kappa(z)\kappa(\sigma z)=K$, and assume
\textup{(B1)}, \textup{(B2)}, \textup{(H1)}, \textup{(H2)}. Then:
\begin{enumerate}[label=\textup{(\alph*)}]
\item the insensitive distributions are exactly the elements of
  $\Sbar_\sigma=\{\pi\in\Dbar:\sqrt K\,\pi_{\sigma z}=\kappa(z)\pi_z\}$,
  all at level $C^\ast$, and $\dim\Sbar_\sigma=\cyc(\sigma)-1$;
\item the response rank is $\rresp=m-\cyc(\sigma)$, the codimension of
  $\Sbar_\sigma$ in the simplex.
\end{enumerate}
\end{theorem}

\begin{proof}
(a) By Proposition~\ref{prop:permweights}, $g_O=\sqrt K$ for every
cycle, so Theorem~\ref{thm:master} applies with the single level
$(1-c)/c=\sqrt K$, and condition (ii) is the defining relation of
$\Sbar_\sigma$. There are $\cyc(\sigma)$ cycles, each contributing one
direction.

(b) Write $v(\gamma)=u(\gamma)-C^\ast s(\gamma)$, which is a positive
multiple of $\psi(\gamma)$. By (B2),
$v_z=\sum_\nu f_\nu W_{\nu z}$ with
$W=TR$, where $R:\mathbb R^Z\to\mathbb R^Z$ is
$Rx=\sum_z x_z[(1-C^\ast)e_z-C^\ast\kappa(z)e_{\sigma z}]$. By (H1),
exactly as in \cite[proof of Theorem~8.8]{PaperVI}, the vectors
$f(\gamma)$ span $\mathbb R^N$, so
$\operatorname{span}\{v(\gamma)\}=\operatorname{im}W^{\mathsf T}$ and
$\rresp=\rank W$. By (B1), $\rank W=\rank R$. Finally, $R=L_{C^\ast}$, and
$R x=0$ if and only if $\sqrt K\,x_{\sigma z}=\kappa(z)x_z$ for all
$z$. This system has a one-dimensional solution space on each cycle,
by the consistency $\prod_O\kappa=\sqrt K^{\,|O|}$. Hence
$\dim\ker R=\cyc(\sigma)$ and $\rresp=m-\cyc(\sigma)$.
\end{proof}

\begin{corollary}[What the involution does]
\label{cor:A2}
Under the hypotheses of Theorem~\ref{thm:perm}, the following are
equivalent:
\begin{enumerate}[label=\textup{(\roman*)}]
\item $\sigma$ is an involution;
\item $\rresp=e(\sigma)$, that is, the response rank equals the number
  of free weight parameters of Proposition~\ref{prop:permweights};
\item every cycle of $\sigma$ contributes at most one direction of
  response.
\end{enumerate}
\end{corollary}

\begin{proof}
A cycle of length $\ell$ contributes $\ell-1$ to
$\rresp=m-\cyc(\sigma)=\sum_O(|O|-1)$, and $1$ or $0$ to $e(\sigma)$
according as $\ell$ is even or odd. The difference
$\ell-1-[\ell\text{ even}]$ is $0$ for $\ell\in\{1,2\}$, equals
$2$ for $\ell=3$, and is at least $2$ for $\ell\ge4$. It is never
negative. So $\rresp=e(\sigma)$ if and only if every cycle has length
at most $2$, which is (i), and also (iii).
\end{proof}

\begin{remark}[Reading]
Corollary~\ref{cor:A2} identifies the work of the involution.
Condition (A2) is not needed for $C^\ast$ to exist, nor for it to be
universal. What it guarantees is that the response is
\emph{orbit-local}, with one direction for each orbit, and that it is
matched one-to-one with the moduli of the weight: each non-trivial
orbit carries exactly one deformation parameter $r_O$ and exactly one
direction of response. For a $3$-cycle, by contrast, the weight is
rigid but the response has rank $2$. For a $4$-cycle there is one
parameter and a response of rank $3$. The statement $\rresp=k$ of
\cite[Corollary~8.9]{PaperVI} is, in this light, the involutive case of
$\rresp=m-\cyc(\sigma)$. The twisted antisymmetry
$v_{\sigma z}=-r_zv_z$ of \cite[Lemma~8.6]{PaperVI}, which makes the
response odd under $\sigma$, is the form this matching takes when
$\sigma^2=\mathrm{id}$.
\end{remark}

\subsection{The full-rank condition (B1)}
\label{subsec:B1}

From now on, (A1)--(A4) are in force. Fix a representative $z_i$ in
each non-trivial orbit, $i=1,\dots,k$, and define the $N\times k$
matrix
\[
M_{\nu i}=\sqrt K\,A_\nu(z_i)-\kappa(z_i)\,A_\nu(\sigma z_i),
\]
with columns $c_i=\sqrt K\,\mathbf A(z_i)-\kappa(z_i)\mathbf A(\sigma z_i)$.

\begin{theorem}[The response rank is the rank of $M$]
\label{thm:rankM}
Assume \textup{(A1)--(A4)}, \textup{(B2)} and \textup{(H1)},
\textup{(H2)}, but not \textup{(B1)}. Then $\rresp=\rank M$. In
particular $\rresp\le k$ unconditionally, with equality if the $2k$
vectors $\mathbf A(z)$, $z$ in non-trivial orbits, are linearly
independent; a fortiori under \textup{(B1)}.
\end{theorem}

This is \cite[Theorem~8.8]{PaperVI}; we recall it because it is the
first half of Theorem~\ref{thm:levelset}. The second statement is
\cite[Corollary~8.9]{PaperVI}, whose proof only uses the independence
of the $2k$ vectors $\mathbf A(z)$ at non-fixed sites. The full-rank
condition enters only to evaluate $\rank M$, and the amplitudes at
fixed points play no role at all. The next theorem identifies the
second meaning of $\rank M$.

\begin{theorem}[The $C^\ast$-level set]
\label{thm:levelset}
Assume \textup{(A1)--(A4)}, \textup{(B2)}, \textup{(H1)},
\textup{(H2)}. Then the set of distributions insensitive at level
$C^\ast$ contains $\Sbar$, and it equals $\Sbar$ if and only if
$\rank M=k$. Consequently, the following are equivalent:
\[
\rresp=k
\quad\Longleftrightarrow\quad
\rank M=k
\quad\Longleftrightarrow\quad
\{\pi\in\Dbar:\ C(\pi,\cdot)\equiv C^\ast\}=\Sbar .
\]
\end{theorem}

\begin{proof}
For $\pi\in\Dbar$ and each orbit representative put
$\eta_i(\pi)=\sqrt K\,\pi_{z_i}-\kappa(\sigma z_i)\,\pi_{\sigma z_i}$.
Then $\pi\in\Sbar$ if and only if $\eta(\pi)=0$. Indeed, the condition
of Definition~\ref{def:insensitive} at $z=\sigma z_i$ is $\eta_i=0$,
at $z=z_i$ it is equivalent to it by (A4), and at fixed points it is
automatic. We compute $L_{C^\ast}\pi$ from \eqref{eq:Lc}. At a fixed
point it vanishes, since $1-C^\ast=C^\ast\sqrt K$. On an orbit,
$(L_{C^\ast}\pi)_{z_i}=C^\ast\eta_i$ and, using
$\kappa(z_i)=K/\kappa(\sigma z_i)$ and
$\kappa(\sigma z_i)=\sqrt K\,r_{z_i}$,
\[
(L_{C^\ast}\pi)_{\sigma z_i}
=C^\ast\bigl[\sqrt K\,\pi_{\sigma z_i}-\kappa(z_i)\pi_{z_i}\bigr]
=-\frac{C^\ast}{r_{z_i}}\,\eta_i .
\]
Therefore
\[
T L_{C^\ast}\pi
=C^\ast\sum_i\eta_i\Bigl[\mathbf A(z_i)-\tfrac1{r_{z_i}}\mathbf
A(\sigma z_i)\Bigr]
=\frac{C^\ast}{\sqrt K}\sum_i\eta_i\,c_i
=\frac{C^\ast}{\sqrt K}\,M\eta(\pi),
\]
because $\kappa(z_i)/\sqrt K=1/r_{z_i}$. By
Lemma~\ref{lem:reduction}, $\pi$ is insensitive at level $C^\ast$ if
and only if $M\eta(\pi)=0$.

If $\rank M=k$, this forces $\eta(\pi)=0$, that is $\pi\in\Sbar$.
If $\rank M<k$, take $0\neq\eta\in\ker M$ and an interior neutral
(which exists by Corollary~\ref{cor:A4} or \cite[Definition~8.2]{PaperVI})
$\pi^0$. Let $x\in\mathbb R^Z$ have $x_{z_i}=\eta_i/\sqrt K$ and zeros
elsewhere, so that $\eta(\pi^0+\varepsilon x)=\varepsilon\eta$. For
$\varepsilon>0$ small, $\pi^0+\varepsilon x$ is positive, and after
normalisation it is a distribution with $\eta\neq0$ in $\ker M$. It
is insensitive at level $C^\ast$ but not neutral. The first
equivalence in the display is Theorem~\ref{thm:rankM}.
\end{proof}

\begin{example}[Maximal response rank without full rank]
\label{ex:B1}
Take one orbit $\{a,b\}$ with $r\neq1$, no fixed points, positive
amplitudes with $\mathbf A(a)=\mathbf A(b)$, and $B$ given by (B2).
Then (B1) fails, but $c_1=(\sqrt K-\kappa(a))\mathbf A(a)\neq0$, so
$\rank M=k=1$. By Theorems~\ref{thm:rankM} and~\ref{thm:levelset},
$\rresp=1$ and the $C^\ast$-level set is exactly $\Sbar$. Nevertheless
$u_a=u_b$, so
\[
C(\pi,\gamma)=\frac{1}{1+\pi_a\kappa(a)+\pi_b\kappa(b)}
\]
is independent of $\gamma$ for \emph{every} $\pi$. Every distribution
is insensitive, at a level that varies with $\pi$ and equals
$C^\ast$ exactly on $\Sbar$.
\end{example}

\begin{remark}[The two layers of (B1)]
\label{rem:B1layers}
The full-rank condition therefore splits into two layers:
\begin{itemize}[nosep]
\item The condition $\rank M=k$, weaker than (B1), secures both the
  maximal response rank and the identification of $\Sbar$ with the
  $C^\ast$-level set (Theorem~\ref{thm:levelset}).
\item Full rank secures, in addition and under (A1)--(A4), that no
  distribution is insensitive at any level other than $C^\ast$
  (Theorem~\ref{thm:master}, since all cycles then have $g_O=\sqrt
  K$).
  Example~\ref{ex:B1} shows that this second layer is not implied by
  the first.
\end{itemize}
This settles the status of (B1). It is independent, and it is
stronger than what the rank theorem uses. It is, however, not stronger
than what the characterisation of neutrality uses: under (A1)--(A4),
by Theorem~\ref{thm:master}, it excludes insensitivity at every level
other than $C^\ast$, which $\rank M=k$ alone does not
(Example~\ref{ex:B1}).
\end{remark}

\subsection{The duality (B2)}
\label{subsec:B2}

Without (B2), insensitivity of a given distribution is cheap.

\begin{proposition}[Neutrality without duality]
\label{prop:cheap}
Let $N\ge2$, $m\ge2$, let $A_\nu>0$, let $\pi$ be an interior
distribution and $\beta>0$. Then there exist positive functions
$B_\nu$ such that $C(\pi,\gamma)=1/(1+\beta)$ for every $\gamma$ at
which $\langle\pi,u(\gamma)\rangle\neq0$ (for every $\gamma$ if
$f_\nu>0$), and
such that no vector $\mathbf B(z)$ is proportional to any $\mathbf
A(w)$. In particular, \textup{(B2)} fails for every choice of
$(\sigma,\kappa)$.
\end{proposition}

\begin{proof}
Let $\mathcal D=\{D:Z\to\mathbb R^N:\ \sum_z\pi_zD(z)=0\}$, and put
$\mathbf B(z)=\beta\mathbf A(z)+D(z)$ with $D\in\mathcal D$. Then
$\langle\pi,s\rangle=(1+\beta)\langle\pi,u\rangle+\sum_\nu f_\nu
\sum_z\pi_zD_\nu(z)=(1+\beta)\langle\pi,u\rangle$, which gives the
stated value of $C(\pi,\gamma)$. Since $m\ge2$, the projection
$D\mapsto D(z)$ maps $\mathcal D$ onto $\mathbb R^N$ for each $z$.
For fixed $z,w$, the condition
$\mathbf B(z)\in\mathbb R\mathbf A(w)$ reads
$D(z)\in\mathbb R\mathbf A(w)-\beta\mathbf A(z)$, an affine line in
$\mathbb R^N$. Its preimage under the surjective linear map
$D\mapsto D(z)$ is an affine subspace of $\mathcal D$, proper because
$N\ge2$. A finite union of proper affine subspaces of the
finite-dimensional space $\mathcal D$ contains no open ball, so every
ball around $0$ contains some $D$ outside all of them. Taking that
ball small enough to keep $\mathbf B=\beta\mathbf A+D>0$, such a $D$
has the required properties.
\end{proof}

\begin{proposition}[Accidental neutrality is isolated]
\label{prop:isolated}
In the setting of Proposition~\ref{prop:cheap}, assume moreover
$N>m$ and \textup{(H1)}. Then there is an open dense subset
$\mathcal G\subset\mathcal D$, whose complement has Lebesgue measure
zero, such that for every $D\in\mathcal G$ for which
$(A,\,B=\beta A+D)$ satisfies \textup{(H2)}, the distribution $\pi$
is the \emph{only} insensitive distribution of $(A,B)$, at any level.
\end{proposition}

\begin{proof}
Write $A$ and $D$ also for the $N\times m$ matrices with columns
$\mathbf A(z)$ and $D(z)$, so that $\mathcal D=\{D:D\pi=0\}$, a linear
space of dimension $N(m-1)$. The first part of the proof of
Lemma~\ref{lem:reduction} does not use (B2): by (H1), $\pi'\in\Dbar$ is
insensitive at level $c$ if and only if $((1-c)A-cB)\pi'=0$, that is,
\begin{equation}
\bigl(tA-cD\bigr)\pi'=0,\qquad t:=1-c(1+\beta).
\label{eq:pencil}
\end{equation}
Put $b=A\pi$; since $A>0$ and $\pi>0$, every entry of $b$ is positive,
and $A\pi'\neq0$ for every $\pi'\in\Dbar$.

\emph{Level $c=0$.} Then \eqref{eq:pencil} reads $A\pi'=0$, which is
impossible.

\emph{Level $c=c_0:=1/(1+\beta)$.} Then $t=0$ and \eqref{eq:pencil}
reads $D\pi'=0$. The set of $D\in\mathcal D$ with $\rank D<m-1$ is a
proper algebraic subset, since $\mathcal D$ contains matrices of rank
$m-1$ (for instance $D=EP$, with $P$ a projection of $\mathbb R^Z$ with
kernel $\mathbb R\pi$ and $E$ injective on its image, which is possible
because $N\ge m-1$). For $D$ outside it, $\ker D=\mathbb R\pi$, and
the only solution in $\Dbar$ is $\pi'=\pi$.

\emph{Levels $c\notin\{0,c_0\}$.} Then $t\neq0$ and, with
$\lambda=c/t\neq0$, \eqref{eq:pencil} says that $A-\lambda D$ has a
non-zero kernel. For $\lambda\neq0$ the map $D\mapsto A-\lambda D$ is
an affine bijection of $\mathcal D$ onto the affine space
$\mathcal Y=\{Y:\ Y\pi=b\}$. We bound the dimension of the set
$\Sigma_b$ of rank-deficient $Y\in\mathcal Y$. If $Yx=0$ with $x\neq0$,
then $x\notin\mathbb R\pi$, because $Y\pi=b\neq0$. For a fixed line
$[x]\neq[\pi]$, the conditions $Yx=0$, $Y\pi=b$ are $2N$ independent
affine conditions, so they cut out an affine space of dimension
$Nm-2N$. Letting $[x]$ range over the projective space of dimension
$m-1$,
\[
\dim\Sigma_b\le (m-1)+Nm-2N=\dim\mathcal Y-(N-m+1).
\]
Hence the semialgebraic set
$\mathcal S=\{(\lambda,D)\in(\mathbb R\setminus\{0\})\times\mathcal D:
\ A-\lambda D\in\Sigma_b\}$ has fibres of dimension at most
$\dim\mathcal D-(N-m+1)$ over each $\lambda$, and therefore
$\dim\mathcal S\le\dim\mathcal D-(N-m)$. Its projection to
$\mathcal D$ has dimension at most $\dim\mathcal D-(N-m)<\dim\mathcal D$,
because $N>m$. Here we use three standard facts on semialgebraic
sets \cite[Sections~2.8 and~9.3]{BCR}: a semialgebraic set
whose fibres over a semialgebraic base of dimension $e$ all have
dimension at most $d$ has dimension at most $d+e$ (a consequence of
Hardt's triviality theorem); the dimension does not increase under a
projection; and the closure of a semialgebraic set has the same
dimension as the set. In particular, a semialgebraic subset of
$\mathcal D$ of dimension less than $\dim\mathcal D$ has a closure
with empty interior and of Lebesgue measure zero.

Let $\mathcal G$ be the complement of the closure of the union of the
exceptional sets of the last two cases. For $D\in\mathcal G$, no level
other than $c_0$ admits an insensitive distribution, and at $c_0$ only
$\pi$ does.
\end{proof}

\begin{remark}[Structural versus accidental neutrality]
\label{rem:structural}
The complement of the finitely many affine subspaces excluded in the
proof of Proposition~\ref{prop:cheap} is open and dense in
$\mathcal D$, and so is $\mathcal G$. Their intersection is therefore
open and dense, and it meets the small ball around $0$ on which
$\mathbf B=\beta\mathbf A+D>0$. Choosing $D$ in this intersection, a
single $D$ realises both conclusions: $B>0$, (B2) fails for every
$(\sigma,\kappa)$, and $\pi$ is insensitive and \emph{isolated},
provided (H2) holds. This is automatic when $f_\nu(\gamma)>0$ for all
$\nu$ and $\gamma$: then $u>0$ and $\sum_\nu f_\nu B_\nu>0$, since
$\mathbf A>0$ and $\mathbf B>0$, so $s>0$.
The condition $N>m$ is sufficient for this generic codimension
argument; for $N\le m$ the count gives no information, and we make no
claim in that case. Under the axioms, by contrast, the insensitive
distributions form the polytope $\Sbar$ of dimension $f+k-1$, at a
level determined by the skeleton alone
(Theorem~\ref{thm:characterisation}). This is the precise sense in
which the axioms do not explain neutrality, which is cheap, but
\emph{structural} neutrality. They explain neutrality that is
organised orbit by orbit, at a common level, independent of the
spectral index, and stable in a family of the expected dimension.
\end{remark}

\section{What the axioms characterise}
\label{sec:characterisation}

\begin{theorem}[Characterisation of neutrality]
\label{thm:characterisation}
Let $\mathcal S$ be a complete SDS satisfying \textup{(H1)} and
\textup{(H2)}. Then a distribution $\pi\in\Dbar$ is insensitive to
$\gamma$ if and only if $\pi\in\Sbar$, and the value is then $C^\ast$.
The set $\Sbar$ is a polytope of dimension $f+k-1$ that meets the open
simplex.
\end{theorem}

\begin{proof}
By (A2)--(A4) every orbit has $g_O=\sqrt K$. Apply
Theorem~\ref{thm:master}. The dimension count is its last assertion
with $n_c=k+f$.
\end{proof}

Together with Section~\ref{sec:roles}, this determines the role of
each axiom. The table collects them. Throughout the table, (A1), (H1)
and (H2) hold and $\sigma$ is a bijection. Conditions (A3) and (B2)
are also imposed, except in the row that concerns that condition. Any
further assumptions are stated in the first column; the cited result
gives the full hypotheses, and each row has been checked against
it. Rows marked $\Leftrightarrow$ are
characterisations, that is, equivalences under those hypotheses; the
other rows describe a role established by one-sided results and
examples.

\begin{center}
\begin{tabular}{@{}>{\raggedright\arraybackslash}p{0.30\textwidth}>{\raggedright\arraybackslash}p{0.44\textwidth}>{\raggedright\arraybackslash}p{0.18\textwidth}@{}}
\toprule
Axiom, with its specific hypotheses & Role or characterisation & Reference\\
\midrule
(A4); given (A2), (B1) & $\Leftrightarrow$ an interior insensitive
  distribution exists & Cor.~\ref{cor:A4}\\
(A2); given pointwise (A4), (B1) & $\Leftrightarrow$ $\rresp$ = number
  of weight moduli & Cor.~\ref{cor:A2}\\
$\rank M=k$; given (A2), (A4), without (B1) & $\Leftrightarrow$
  $\rresp=k$ $\Leftrightarrow$ $C^\ast$-level set $=\Sbar$
  & Thm.~\ref{thm:levelset}\\
(B1); given (A2), (A4) & under (A2) and (A4), excludes insensitivity
  at levels other than $C^\ast$; Ex.~\ref{ex:B1} shows what can happen
  when (B1) fails & Thm.~\ref{thm:master}, Ex.~\ref{ex:B1}\\
(B2); with $\mathbf A>0$, $N,m\ge2$ ($N>m$ for isolation), no
  $(\sigma,\kappa)$ fixed &
  without it, insensitivity of one distribution is cheap and
  generically isolated
  & Props.~\ref{prop:cheap}, \ref{prop:isolated}\\
(A3); given (A2), (A4) & selects the positive root; levels in $(0,1)$
  & Sec.~\ref{subsec:A3}\\
\bottomrule
\end{tabular}
\end{center}

Figure~\ref{fig:axioms} displays the same information as a graph,
which makes the dependencies easier to follow.

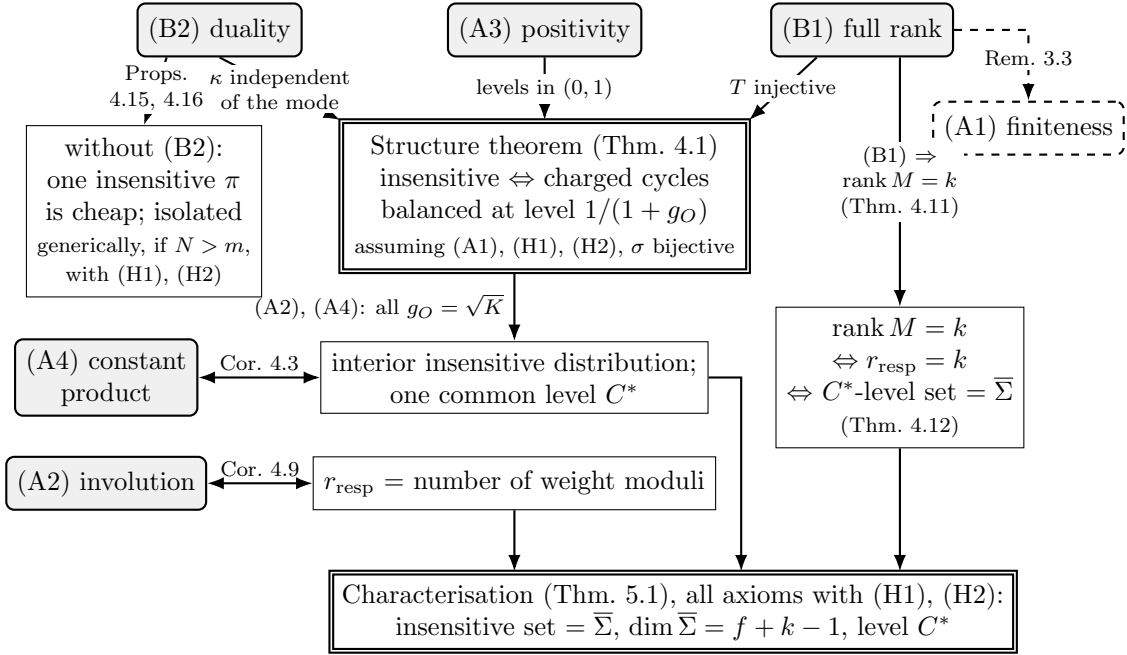
\begin{figure}[t]
\centering
\begin{tikzpicture}[
  font=\small,
  ax/.style={draw,rounded corners=3pt,thick,align=center,fill=gray!12,minimum height=7mm,inner xsep=4pt},
  res/.style={draw,align=center,minimum height=7mm,inner sep=4pt},
  mech/.style={draw,double,thick,align=center,inner sep=5pt},
  imp/.style={-{Latex[length=2.2mm]},thick},
  eqv/.style={{Latex[length=2.2mm]}-{Latex[length=2.2mm]},thick},
  role/.style={-{Latex[length=2.2mm]},thick,dashed},
  lab/.style={font=\scriptsize,fill=white,inner sep=1.5pt,align=center}
]
\node[ax] (B2) at (-4.6,0) {(B2) duality};
\node[ax] (A3) at (-0.3,0) {(A3) positivity};
\node[ax] (B1) at (3.9,0) {(B1) full rank};
\node[mech] (T) at (-0.3,-2.2) {Structure theorem (Thm.~\ref{thm:master})\\ insensitive $\Leftrightarrow$ charged cycles\\ balanced at level $1/(1+g_O)$\\[1pt] {\scriptsize assuming (A1), (H1), (H2), $\sigma$ bijective}};
\draw[imp] (B2.south) -- node[lab]{$\kappa$ independent\\ of the mode} (T.north west);
\draw[imp] (A3.south) -- node[lab]{levels in $(0,1)$} (T.north);
\draw[imp] (B1.south) ++(-0.6,0) -- node[lab]{$T$ injective} (T.north east);
\node[res] (R0) at (-5.6,-2.4) {without (B2):\\ one insensitive $\pi$\\ is cheap; isolated\\ {\scriptsize generically, if $N>m$,}\\[-1pt] {\scriptsize with (H1), (H2)}};
\draw[role] (B2.south) ++(-0.7,0) -- node[lab,pos=0.45]{Props.\\ \ref{prop:cheap}, \ref{prop:isolated}} (R0.north);
\node[ax,dashed,fill=none] (A1) at (6.1,-1.3) {(A1) finiteness};
\draw[role] (B1.east) -| node[lab,pos=0.7]{Rem.~\ref{rem:finiteness}} (A1.north);
\node[ax] (A4) at (-6.1,-4.6) {(A4) constant\\ product};
\node[res] (R2) at (-0.7,-4.6) {interior insensitive distribution;\\ one common level $C^\ast$};
\node[ax] (A2) at (-6.1,-6.0) {(A2) involution};
\node[res] (R3) at (-0.7,-6.0) {$r_{\mathrm{resp}}$ = number of weight moduli};
\draw[eqv] (A4) -- node[lab,above=1pt,fill=none]{Cor.~\ref{cor:A4}} (R2);
\draw[eqv] (A2) -- node[lab,above=1pt,fill=none]{Cor.~\ref{cor:A2}} (R3);
\draw[imp] (T.south -| R2.north) -- node[lab,left=2pt,fill=none]{(A2), (A4): all $g_O=\sqrt K$} (R2.north);
\node[res] (R4) at (4.4,-4.6) {$\operatorname{rank}M=k$\\ $\Leftrightarrow r_{\mathrm{resp}}=k$\\ $\Leftrightarrow$ $C^\ast$-level set $=\overline\Sigma$\\ {\scriptsize (Thm.~\ref{thm:levelset})}};
\coordinate (b1s) at ($(B1.south)+(0.5,0)$);
\draw[imp] (b1s) -- node[lab,pos=0.5]{(B1) $\Rightarrow$\\ $\operatorname{rank}M=k$\\ (Thm.~\ref{thm:rankM})} (b1s |- R4.north);
\node[res,double,thick] (R5) at (1.4,-7.7) {Characterisation (Thm.~\ref{thm:characterisation}), all axioms with (H1), (H2):\\ insensitive set $=\overline\Sigma$, $\dim\overline\Sigma=f+k-1$, level $C^\ast$};
\draw[imp] (R2.east) -- (2.3,-4.6) -- (2.3,-7.7 |- R5.north);
\draw[imp] (R3.south) -- (R3.south |- R5.north);
\draw[imp] (R4.south) -- (R4.south |- R5.north);
\end{tikzpicture}
\caption{How the axioms act. Shaded boxes are the axioms; the
double-framed boxes are the structure theorem and the characterisation.
Single arrows are implications, labelled by the property they supply
or by the result that proves them. Double-headed arrows are
equivalences, under the hypotheses of the cited result. Dashed arrows
record what happens without (B2), and that the finiteness half of
(A1) follows from (B1).}
\label{fig:axioms}
\end{figure}

\begin{remark}[What is not claimed]
\label{rem:notclaimed}
It would be tempting to summarise this section by saying that the
axioms are necessary and sufficient for the invariance phenomenon. In
that form the statement is false. Proposition~\ref{prop:cheap} shows
that the insensitivity of a single distribution requires none of the
axioms, and Theorem~\ref{thm:master} shows that structural
insensitivity survives the failure of (A2) and of the pointwise form
of (A4). The correct statements are those of the table: three
equivalences, each between an axiom (or the condition $\rank M=k$)
and a specific structural property of the insensitive set or of the
response, in the presence of the other axioms; and three roles, for
(B1), (B2) and (A3), established by one-sided results and examples.
\end{remark}

\section{What the axioms do not determine: orientation}
\label{sec:orientation}

The axioms fix the invariant value $C^\ast$, the separatrix $\Sbar$
and the response rank. A natural further question is whether they also
fix the \emph{orientation} of the response, that is, the sign of
$\partial_\gamma C(\pi,\gamma)$ for $\pi\notin\Sbar$. A positive answer,
in a specific form, was conjectured in the first version of Paper~III
\cite{PaperIII} and withdrawn in its revised version. The conjecture
stated that this sign coincides with the sign of
$\langle\pi-\pi^\ast,\psi(\gamma)\rangle$, so that every distribution
on one side of the separatrix responds monotonically to resetting, with
a sign determined by the side. We show that the answer is negative,
and we determine exactly, in the smallest case, when the two criteria
agree.

Consider the skeleton $Z=\{1,2\}$, $\sigma=(1\,2)$, $\kappa\equiv1$,
so that $K=1$, $C^\ast=\tfrac12$, $(k,f)=(1,0)$ and $\Sbar=\{\pi^\ast\}$
with $\pi^\ast=(\tfrac12,\tfrac12)$. Take $A=I_2$, so that (B1) holds,
and $B$ given by (B2), i.e.\ the antidiagonal permutation matrix. Let
$f_1,f_2>0$ be linearly independent spectral functions, differentiable
on $(0,1)$. Then
\[
u=(f_1,f_2),\qquad s=(f_1+f_2,\;f_1+f_2),
\]
and we write $R=f_1/f_2$ and $\pi=(\alpha,1-\alpha)$.

\begin{proposition}[The two criteria in closed form]
\label{prop:orient}
With the notation above,
\[
\partial_\gamma C(\pi,\gamma)
=\frac{(2\alpha-1)\,f_2^{2}\,R'}{(f_1+f_2)^{2}},
\qquad
\langle\pi-\pi^\ast,\psi\rangle
=\frac{(2\alpha-1)\,f_2\,(R-1)}{2\,(f_1+f_2)} .
\]
Hence
$\operatorname{sgn}\partial_\gamma C=\operatorname{sgn}(2\alpha-1)\,
\operatorname{sgn}R'$ and
$\operatorname{sgn}\langle\pi-\pi^\ast,\psi\rangle
=\operatorname{sgn}(2\alpha-1)\,\operatorname{sgn}(R-1)$.
\end{proposition}

\begin{proof}
$C=(\alpha f_1+(1-\alpha)f_2)/(f_1+f_2)$. Differentiating, the
numerator is $\alpha(f_1'f_2-f_1f_2')+(1-\alpha)(f_2'f_1-f_2f_1')
=(2\alpha-1)(f_1'f_2-f_1f_2')$, and $f_1'f_2-f_1f_2'=f_2^2R'$. For the
second identity,
$\psi=(u-\tfrac12s)/\langle\pi^\ast,s\rangle
=\bigl(\tfrac{f_1-f_2}{2},\tfrac{f_2-f_1}{2}\bigr)/(f_1+f_2)$, and
$\pi-\pi^\ast=(\alpha-\tfrac12)(1,-1)$.
\end{proof}

\begin{theorem}[Coincidence criterion]
\label{thm:orient}
In the setting above, the two criteria predict the same sign at
$\gamma$ for every non-neutral $\pi$ if and only if
$\operatorname{sgn}R'(\gamma)=\operatorname{sgn}(R(\gamma)-1)$. Where
$R(\gamma)\neq1$ and $R'(\gamma)\ne0$, this says that they agree where $R$
moves away from $1$ and disagree where it moves towards $1$. At a
point with $R(\gamma)=1$ the geometric criterion gives $0$, and the
criteria agree only if $R'(\gamma)=0$ as well.
\end{theorem}

\begin{proof}
Immediate from Proposition~\ref{prop:orient}, since the factor
$2\alpha-1$ is common to both and non-zero off $\Sbar$.
\end{proof}

\begin{corollary}[Orientation is not determined by the static data]
\label{cor:orient}
\leavevmode
\begin{enumerate}[label=\textup{(\alph*)}]
\item The conjecture fails: every pair $f_1,f_2$ with $R$ increasing
  and $R<1$ on an interval, or $R$ decreasing and $R>1$, gives a
  countermodel on that interval. Examples are
  $(f_1,f_2)=(1+\gamma,2)$, $(2+\gamma,3)$, $(1+2\gamma,3)$ and
  $(1+\gamma^2,2)$, which give opposite signs on all of $(0,1)$.
\item Disagreement can be local. For $(f_1,f_2)=(e^\gamma,2)$ the
  criteria disagree for $\gamma<\log2$ and agree for $\gamma>\log2$.
\item The skeleton and the amplitudes do not determine the
  orientation. The spectral families $(1+\gamma,2)$ and $(2,1+\gamma)$
  share $\mathcal E$, $A$ and $B$, and give opposite signs of
  $\partial_\gamma C$ for every non-neutral $\pi$.
\end{enumerate}
\end{corollary}

\begin{remark}[Attribution and reading]
\label{rem:orient}
That the orientation conjecture of the first version of \cite{PaperIII} fails was first
observed by T.~Newton \cite{Newton}, whose counterexample is the first
pair in Corollary~\ref{cor:orient}(a). The coincidence criterion of
Theorem~\ref{thm:orient} was obtained independently afterwards. The
two-site theorem of \cite{PaperIII}, which gives the sign of
$\partial_\gamma C$ as
$\operatorname{sgn}(\alpha-\alpha^\ast)\operatorname{sgn}R'$, is
correct and is confirmed by Proposition~\ref{prop:orient}. What fails
is only the identification of that sign with the geometric one.

The reading is the one announced. The axioms determine the static
data: the value $C^\ast$, the position of the separatrix and the
dimension of the response. The orientation of the response depends on
the spectral functions, through the position of $R$ relative to $1$ and
the direction in which it moves, and Corollary~\ref{cor:orient}(c)
shows that it is not a function of $(\mathcal E,A,B)$. Orientation is
\emph{dynamical} information, a layer on top of the structure and not
part of it.
\end{remark}

\section{Internal rigidity}
\label{sec:rigidity}

We now look inside a single skeleton $\mathcal E=(Z,\sigma,\kappa)$ and
ask which of the constructions familiar from vector spaces, groups or
modules it supports: internal laws of composition, quotients and
symmetries. The answer is that it supports all of them, but only in
degenerate forms that are completely determined by the
classification data of \cite[Theorem~7.2]{PaperVI}.

\subsection{Compatible internal operations}

\begin{theorem}[The weight obstructs internal operations]
\label{thm:rigidity}
Let $\mathcal E$ be a skeleton of degree $K$, and suppose $Z$ carries
an associative operation $(z,w)\mapsto z\cdot w$ for which the weight
is compatible, in the sense that
\begin{equation}
\sqrt K\,\kappa(z\cdot w)=\kappa(z)\,\kappa(w)
\qquad\text{for all }z,w\in Z .
\label{eq:compatible}
\end{equation}
Then $\kappa\equiv\sqrt K$ and all orbit ratios are $1$.
\end{theorem}

\begin{proof}
Write $\kappa=\sqrt K\,e^{-\lambda}$. Condition \eqref{eq:compatible}
says that $\lambda:(Z,\cdot)\to(\mathbb R,+)$ is a homomorphism of
semigroups. For each $z$, the powers $z,z^2,z^3,\dots$ lie in the
finite set $Z$, so $z^a=z^b$ for some $a<b$. Then
$a\lambda(z)=b\lambda(z)$, hence $\lambda(z)=0$.
\end{proof}

The theorem uses neither the involution nor any relation between the
operation and $\sigma$. The natural compatible structures are groups
in which $\sigma$ is the inversion, and for them there is an
additional arithmetic obstruction, independent of the weight.

\begin{proposition}[Abelian group structures]
\label{prop:groups}
Let $(Z,\sigma)$ be a finite set with an involution, $|Z|=m$, with $f$
fixed points. There is an abelian group structure on $Z$ in which
$\sigma(z)=-z$ if and only if $f=2^j$, where $j=0$ if $m$ is odd and
$1\le j\le v_2(m)$ if $m$ is even. Here $v_2$ is the $2$-adic
valuation.
\end{proposition}

\begin{proof}
In such a structure the fixed points of $\sigma$ are the elements with
$2z=0$. They form the $2$-torsion subgroup, which is elementary
abelian of rank $j$ equal to the number of cyclic factors of even
order in the primary decomposition of the group. The $2$-primary
part is $\bigoplus_{i=1}^j\mathbb Z/2^{a_i}$ with $a_i\ge1$ and $\sum
a_i=v_2(m)$. So $j\le v_2(m)$, $j\ge1$ when $m$ is even (by Cauchy's
theorem), and $j=0$ when $m$ is odd. Conversely, given such $j$,
choose $a_1,\dots,a_j\ge1$ with sum $v_2(m)$, and let $G$ be the
product of the corresponding cyclic $2$-groups with a cyclic group of
order $m/2^{v_2(m)}$. Then $G$ has exactly $2^j=f$ elements with
$g=-g$ and $(m-f)/2$ pairs $\{g,-g\}$. Any bijection $Z\to G$ sending
fixed points to $2$-torsion and orbits to pairs transports the group
structure, with $\sigma$ becoming the inversion.
\end{proof}

\begin{corollary}
\label{cor:f0}
A skeleton with $f=0$ admits no group structure, abelian or not, in
which $\sigma$ is the inversion, because the neutral element is
always fixed. For abelian groups, Proposition~\ref{prop:groups}
gives: for $m=6$ the only possible value of $f$ is $2$
(realised by $\mathbb Z/6$), for $m=4$ the possible values are $2$ and
$4$, and for $m=12$ they are $2$ and $4$. In particular $m=12$,
$f=6$ is impossible.
\end{corollary}

\begin{remark}
For non-abelian groups the inversion is still an involution, but the
count of its fixed points, the elements of order at most $2$, is not
restricted to powers of $2$. The dihedral group of order $10$ has
$f=6$. Theorem~\ref{thm:rigidity} applies in every case: whatever
group is chosen, a compatible weight is constant. The arithmetic of
Proposition~\ref{prop:groups} is not visible in the classification of
\cite{PaperVI}, which records $f$ but not its arithmetic.
\end{remark}

\subsection{Congruences and quotients}

\begin{definition}
A \emph{congruence} on a skeleton $\mathcal E$ is an equivalence
relation $\sim$ on $Z$ such that $z\sim w$ implies $\sigma z\sim\sigma
w$ and $\kappa(z)=\kappa(w)$.
\end{definition}

\begin{proposition}[Quotients]
\label{prop:quotients}
\leavevmode
\begin{enumerate}[label=\textup{(\alph*)}]
\item If $\sim$ is a congruence, then $Z/{\sim}$ with the induced
  $\bar\sigma$ and $\bar\kappa$ is a skeleton of the same degree, and
  the projection is a surjective morphism.
\item Every surjective morphism $\phi:\mathcal E\to\mathcal E'$
  factors as the projection onto $\mathcal E/{\ker\phi}$ followed by
  an isomorphism, where $z\ker\phi\,w$ if and only if
  $\phi(z)=\phi(w)$.
\item There is a largest congruence, $z\equiv_\kappa w$ if and only if
  $\kappa(z)=\kappa(w)$. The quotient $\mathcal E/{\equiv_\kappa}$ is
  isomorphic to the \emph{reduced skeleton} on the set of values
  $V=\kappa(Z)\subset\mathbb R_{>0}$, with involution $v\mapsto K/v$
  and weight the identity.
\end{enumerate}
\end{proposition}

\begin{proof}
(a) $\bar\sigma$ and $\bar\kappa$ are well defined by the two
conditions, $\bar\sigma^2=\mathrm{id}$, $\bar\kappa>0$, and
$\bar\kappa([z])\bar\kappa(\bar\sigma[z])=\kappa(z)\kappa(\sigma z)=K$.
The projection commutes with the involutions and preserves weights by
construction. (b) $\ker\phi$ is a congruence because $\phi$ commutes
with $\sigma$ and preserves $\kappa$. The induced map
$Z/{\ker\phi}\to Z'$ is a bijective morphism, and its inverse
preserves $\sigma$ and $\kappa$, so it is an isomorphism
\cite[Definition~7.1]{PaperVI}. (c) $\equiv_\kappa$ is a congruence,
because $\kappa(z)=\kappa(w)$ implies
$\kappa(\sigma z)=K/\kappa(z)=K/\kappa(w)=\kappa(\sigma w)$. Every
congruence is contained in it by definition. The map $[z]\mapsto
\kappa(z)$ identifies the quotient with $V$, and $\bar\sigma$ becomes
$v\mapsto K/v$.
\end{proof}

\begin{remark}
A congruence may identify the two sites of an orbit only when
$r_O=1$, in which case the orbit becomes a fixed point of the
quotient \cite[Remark~3.6]{PaperVI}. The congruences of $\mathcal E$
are therefore the equivalence relations refining $\equiv_\kappa$ and
compatible with $\sigma$. On each pair of level sets
$\{\kappa=v\}\cup\{\kappa=K/v\}$ with $v\neq\sqrt K$, they are the
$\sigma$-equivariant partitions, and on $\{\kappa=\sqrt K\}$ they are
the $\sigma$-stable ones. Since $\kappa$ takes at most $2k+1$ values
and is generically injective on non-fixed sites, a generic skeleton
has few congruences. This is the same rigidity again.
\end{remark}

\subsection{Automorphisms}

\begin{proposition}[Automorphism group]
\label{prop:aut}
Let $\mathcal E$ have $f$ fixed points, $n_1$ orbits with $r_O=1$, and,
for each class $[r]\neq[1]$, $n_{[r]}$ orbits of class $[r]$. Then
\[
\Aut(\mathcal E)\;\cong\;S_f\times\bigl(S_2\wr S_{n_1}\bigr)\times
\prod_{[r]\neq[1]}S_{n_{[r]}},
\qquad
|\Aut(\mathcal E)|=f!\;2^{n_1}n_1!\prod_{[r]\neq[1]}n_{[r]}! .
\]
\end{proposition}

\begin{proof}
An automorphism commutes with $\sigma$, so it maps fixed points to
fixed points and orbits to orbits. It preserves $\kappa$, so it maps
an orbit to an orbit of the same class $[r]$. Since
$\kappa=(\sqrt K/r,\sqrt K\,r)$ on an orbit, when $r\neq1$ the two
sites carry different weights, and an orbit of class $[r]$ (ratio $r$
or $1/r$ at its representative) carries the same two values; so once
the image orbit is chosen, the image of each site is forced. When
$r=1$, $\kappa$ is constant on the orbit and the two sites may be
exchanged, which produces the factor $2$ for each orbit of class
$[1]$ and no such factor for the other classes. Conversely, any such
choice defines an automorphism.
\end{proof}

\begin{corollary}
$\Aut(\mathcal E)$ is trivial if and only if $f\le1$, $n_1=0$, and
the classes $[r_O]$ are pairwise distinct. This is the generic
situation.
\end{corollary}

\subsection{Reading}

The three results say the same thing. The weight is the obstruction
to internal operations (Theorem~\ref{thm:rigidity}), the obstruction
to identifying sites (Proposition~\ref{prop:quotients}), and the
obstruction to symmetry (Proposition~\ref{prop:aut}). In each case
the answer is computed from the classification data
$(K,k,f,\{[r_O]\})$ of \cite[Theorem~7.2]{PaperVI}, and no new
invariant appears. This is a result and not a shortcoming. When a
structure is classified in its founding paper, the analogue of the
programme ``subobjects, quotients, dependence, generation'' that is so
fruitful for vector spaces exhausts itself quickly, and the richness
lies in the class. There, the product splits orbits, and the
resulting arithmetic is non-trivial \cite[Remark~4.7]{PaperVI}.

\section{Conclusions}
\label{sec:conclusions}

The axioms of a spectral duality structure are independent, and each
of them does identifiable work. The analysis, however, relocates that
work. The mechanism of insensitivity does not come from the involution,
nor from the pointwise product condition. It comes from a duality
whose weight does not depend on the spectral index, which balances
each cycle at the level $1/(1+g_O)$ given by the geometric mean of the
weight (Theorem~\ref{thm:master}). The universal value $C^\ast$ is what
this level becomes when (A2) and (A4) make $g_O=\sqrt K$ for every
orbit. On top of this mechanism:
\begin{itemize}
\item the constant product makes all levels coincide, which is
  equivalent to the existence of an interior insensitive
  distribution;
\item the involution matches the response with the moduli, one
  direction and one parameter per orbit;
\item positivity selects the root;
\item the full-rank condition splits into a part that fixes the
  response rank and the separatrix, the condition $\rank M=k$, and a
  part that excludes insensitivity at any other level;
\item without the duality, the insensitivity of a single distribution
  is cheap and, for $N>m$, generically isolated; with it, and the
  other axioms,
  neutrality is structural.
\end{itemize}
What the axioms do \emph{not} fix is the orientation of the response.
It depends on the spectral functions and not on the skeleton or the
amplitudes (Corollary~\ref{cor:orient}), so it is dynamical
information on top of the static structure.
One of the original hypotheses, full rank, turned out to be stronger
than the response rank theorem needs, and the theorem gained scope
when the condition actually used was identified
(Theorem~\ref{thm:rankM}). This is the kind of outcome an axiomatic
analysis should produce. Another, the involution, turned out not to
be needed for the invariance at all, and its true role is a statement
about the response.

A single structure, finally, is rigid. Compatible internal operations
force a constant weight, quotients and symmetries are computed from the
classification data, and the absence of fixed points excludes any
group structure. The structure worth studying is that of the class,
and of the objects that realise it; Section~\ref{sec:further} lists the
questions that this paper leaves open.

\section{Further directions}
\label{sec:further}

The object of this paper is the axiomatics of a single spectral
duality structure. The following questions lie outside its scope, but
are suggested by it.

\paragraph{The class.} Section~\ref{sec:rigidity} shows that the
internal algebra of one structure is rigid, and that its richness lies
in the class. The natural continuation is the monoid of isomorphism
classes under the Cartesian product of \cite{PaperVI}: its
factorisation theory, its irreducible elements and its Grothendieck
group. This will be the subject of a subsequent paper.

\paragraph{Realisability.} The axioms were extracted from random walks
with geometric resetting, but nothing in them refers to probability.
Two questions follow. The first is which structures are realised by
concrete objects, inside probability (positive Markovian resetting
processes) and outside it, and whether the realisable structures form
a subclass closed under the operations of \cite{PaperVI}. The second
is what the axiomatic analysis says about a given realised class: which
of the axioms are automatic in it, and what the general results of
Sections~\ref{sec:roles}--\ref{sec:characterisation} imply for it
without further computation.

\paragraph{Weighted permutations.} Theorem~\ref{thm:perm} shows that
the invariance mechanism survives when the involution is replaced by
an arbitrary permutation whose cycles are balanced at a common level.
Such weighted permutations form a natural enlargement of the class, in
which the involutive structures are characterised by the matching of
response and moduli (Corollary~\ref{cor:A2}). Whether the algebra of
\cite{PaperVI} extends to this enlargement is open.

\paragraph{The exact role of full rank.} Under (A1)--(A4), (B2), (H1)
and (H2), full rank is sufficient to exclude insensitivity at levels
other than $C^\ast$, and Example~\ref{ex:B1} shows that $\rank M=k$ is
not. Under the same hypotheses without (B1), Lemma~\ref{lem:reduction}
shows that the exact condition is that
$TL_c\pi\neq0$ for every $\pi\in\Dbar$ and every $c\neq C^\ast$; a
description of this condition in terms of the amplitudes is open.
Likewise, Proposition~\ref{prop:isolated} is proved for $N>m$, and the
case $N\le m$ is open.

\paragraph{Orientation.} Theorem~\ref{thm:orient} settles the agreement
between the geometric and analytic sign criteria in the two-site case.
A criterion for general structures is open.

\section*{Acknowledgements}

The author thanks Tony Newton for stimulating discussions throughout
the development of this research programme, and in particular for the
counterexample discussed in Section~\ref{sec:orientation}. He is also
grateful to several colleagues whose comments on this and related work
helped to improve it.

The author used AI tools (Claude and ChatGPT) for English-language editing and numerical checks during the preparation of this work. The author checked all mathematical statements and proofs and takes full responsibility for the content.

\end{document}